\documentclass[11pt,reqno]{amsart}

\usepackage{amsmath,amssymb,mathrsfs}
\usepackage{graphicx,cite,times}

\newtheorem{theo}{Theorem}[section]

\newtheorem{lemm}[theo]{Lemma}

\numberwithin{equation}{section}

\begin{document}

\title{Analysis of time-domain scattering by periodic structures}

\author{Yixian Gao}
\address{School of Mathematics and Statistics, Center for Mathematics
and Interdisciplinary Sciences, Northeast Normal University, Changchun,
Jilin 130024, P.R.China}
\email{gaoyx643@nenu.edu.cn}

\author{Peijun Li}
\address{Department of Mathematics, Purdue University, West Lafayette, IN 47907,
USA.}
\email{lipeijun@math.purdue.edu}

\thanks{The research of YG was partially supported by NSFC grant 11571065 and
Jilin Science and Technology Development Project. The research of PL was
supported in part by the NSF grant DMS-1151308.}


\begin{abstract}
This paper is devoted to the mathematical analysis of a time-domain
electromagnetic scattering by periodic structures which are known as
diffraction gratings. The scattering problem is reduced
equivalently into an initial-boundary value problem in a bounded domain by using
an exact transparent boundary condition. The well-posedness and stability of the
solution are established for the reduced problem. Moreover, a priori energy
estimates are obtained with minimum regularity requirement for the data and
explicit dependence on the time. 
\end{abstract}
\maketitle

\section{Introduction}

This paper is concerned with the mathematical analysis of an electromagnetic
scattering problem in periodic structures, where the wave propagation is
governed by the time-domain Maxwell equations. The scattering theory in periodic
diffractive structures, also known as diffraction gratings, has applications in
many cutting-edge scientific areas including ultra-fast and high-energy
lasers, space flight instruments, astronomy, and synchrotron spectrometers. A
good introduction can be found in \cite{Petit2013} to diffraction grating
problems and various numerical approaches. The book \cite{Friedman1990} contains
descriptions of several mathematical problems that arise in diffractive optics
modeling in industry. Some more recent developments are addressed
in \cite{BaoCowsar2001} on theory, analysis, and computational techniques of
diffractive optics. 

The time-harmonic grating problems have been extensively studied by
many researchers via either the integral equation methods or the
variational methods \cite{Chen1991, Dobson1992, Bao1995Finite, Bao1996Numerical,
Bao1997}. A survey may be found in \cite{BaoandDobson1995} for mathematical
studies in rigorous grating theory. The general result may be stated as follows:
The diffraction problem has a unique solution for all but a countable sequence
of singular frequencies. Unique solvability for all frequencies can be obtained
for gratings which have absorbing media or perfectly electrically conducting
surfaces with Lipschitz profiles. Numerical methods are developed for both the
two-dimensional Helmholtz equation (one-dimensional gratings) and the
three-dimensional Maxwell equations (crossed or two-dimensional gratings)
\cite{BaoYang2000, BaoChenWu2005, ChenWu2003Adaptive, BaoLiWu2010, WangWu2015}.

The time-domain scattering problems have attracted considerable attention due to
their capability of capturing wide-band signals and modeling more general
material and nonlinearity \cite{ChenMonk2014, Jin2009, JiHuang2013,
Riley2008, WangWang2012}. Comparing with the time-harmonic problems, the
time-domain problems are much less studied due to the additional challenge of
the temporal dependence. Rigorous mathematical analysis is very rare. The
analysis can be found in \cite{WangWang2014, Chen2008Maxwell} for the
time-domain acoustic and electromagnetic obstacle scattering problems. We refer
to \cite{LiWangWood2015} for the analysis of the time-dependent electromagnetic
scattering from a three-dimensional open cavity. Numerical solutions can be
found in \cite{Li2015, Veysoglu1993} for the time-dependent wave scattering by
periodic structures/surfaces. The theoretical analysis is still lacking for the
time-domain scattering by periodic structures. 

The goal of this work is to analyze mathematically the time-domain scattering
problem which arises from the electromagnetic wave propagation in a periodic
structure. Specifically, we consider an electromagnetic plane wave which is
incident on a one-dimensional grating in $\mathbb{R}^3$. So the structure is
assumed to be invariant in the $y$-direction and periodic in the $x$-direction.
The three-dimensional Maxwell equations can be decomposed into two fundamental
polarizations: transverse electric (TE) polarization and transverse magnetic
(TM) polarization, where Maxwell's equations are reduced to the two-dimensional
wave equation. We shall study the wave equation in two dimensions for both
polarizations. The structure can also be characterized by the medium parameters:
the electric permittivity and the magnetic permeability. They are periodic in
$x$ and assumed only to be bounded measurable functions. Hence our method works
for very general gratings whose surfaces/interfaces are allowed to be
Lipschitz profiles or even graphs of some Lipschitz continuous functions. 

There are two challenges of the problem: time dependence and unbounded
domain. In the frequency domain, various approaches have been developed to
truncate unbounded domains into bounded ones, such as absorbing boundary
conditions (ABCs), transparent boundary conditions (TBCs), and perfectly matched
layer (PML) techniques. These effective boundary conditions are being extended
to handle time-domain problems \cite{Alpert2002, Chen2009Convergence, Grote1995,
Hagstrom1999}. Utilizing the Laplace transform as a bridge between the
time-domain and the frequency domain, we develop an exact time-domain TBC and
reduce the problem equivalently into an initial boundary value problem in a
bounded domain. Using the energy method with new energy functions, we show the
well-posedness and stability of the time-dependent problem. The
proofs are based on examining the well-posedness of the time-harmonic Helmholtz
equations with complex wavenumbers and applying the abstract inversion theorem
of the Laplace transform. Moreover, a priori estimates, featuring an explicit
dependence on time and a minimum regularity requirement of the data, are
established for the wave field by studying directly the time-domain wave
equation.

The paper is organized as follows.  In section \ref{FRP}, we introduce the
model problem and develop a TBC to reduce it into an initial boundary value
problem. Section \ref{SRP} is devoted to the analysis of the reduced problem,
where the well-posdeness and stability are addressed and a priori estimates are
provided. We conclude the paper with some remarks and directions for future work
in section \ref{CL}.

\section{Problem formulation}\label{FRP}

In this section, we introduce the mathematical model of interest and develop an
exact TBC to reduce the scattering problem from an unbounded domain into a
bounded domain. 

\subsection {A model problem}

Consider the  system of time-domain Maxwell equations in $\mathbb{R}^3$ for
$t>0$:
\begin{equation}\label{TMW}
\begin{cases}
\nabla \times \boldsymbol {E}(x, y, z, t) +\mu \partial_t \boldsymbol {H}(x, y,
z, t)=0,\\
\nabla \times \boldsymbol {H} (x, y,  z, t)-\varepsilon \partial_t \boldsymbol
{E}(x, y, z, t)=0,
 \end{cases}
\end{equation}
where $\boldsymbol E$ is the electric field, $\boldsymbol H$ is the magnetic
field, $\varepsilon$ and $\mu$ are the dielectric permittivity and magnetic
permeability, respectively, and satisfy 
\[
0 < \varepsilon_{\rm min} \leq \varepsilon \leq \varepsilon_{\rm max}<
\infty,\quad 0 < \mu_{\rm min} \leq \mu \leq \mu_{\rm max} < \infty.
\]
Here $\varepsilon_{\rm min}, \varepsilon_{\rm max}, \mu_{\rm min}, \mu_{\rm
max}$ are constants. We assume that the structure is invariant in the
$y$-direction and thus focus on the one-dimensional grating. The more
complicated problem in biperiodic structures will be considered in a separate
work. There are two fundamental polarizations for the one-dimensional structure:

(i) TE polarization.  The electric and magnetic fields are 
\[
 \boldsymbol E(x, y, z, t)=[0, E(x,  z, t),0 ]^{\top},\quad \boldsymbol H(x,  z,
t) =[H_1(x,  z, t), 0, H_3(x,  z, t)]^{\top}. 
\]
Eliminating the magnetic field from (\ref{TMW}), we get the wave equation for
the electric field:
\begin{equation}\label{TE}
\varepsilon \partial_t^2 E(x, z, t)= \nabla \cdot (\mu^{-1} \nabla E(x, z, t)).
\end{equation}

(ii) TM polarization.  The electric and magnetic fields are
\[
 \boldsymbol E(x, y, z, t)=[E_1(x, z, t), 0, E_3(x, z,
t)]^{\top},\quad \boldsymbol H(x, y, z, t)= [0, H(x, z, t), 0]^{\top}. 
\]
We may eliminate the electric field from \eqref{TMW} and obtain the wave
equation for the magnetic field:
\begin{equation}\label{TM}
\mu \partial_{t}^2 H(x, z, t)= \nabla \cdot (\varepsilon^{-1} \nabla H(x, z,
t)).
\end{equation}
It is clear to note from \eqref{TE} and \eqref{TM} that the TE and TM
polarizations can be handled in a unified way by formally exchanging the roles
of $\varepsilon$ and $\mu$. We will just present the results by using \eqref{TE}
as the model equation in the rest of the paper. 

Now let us specify the problem geometry, which is shown in Figure \ref{pg}.
Since the structure and medium are assumed to be periodic in the $x$ direction,
there exists a period $\Lambda>0$ such that 
\[
\varepsilon (x+ n\Lambda, z) =\varepsilon(x, z),\quad \mu(x+n\Lambda, z)= \mu(x,
z),\quad\forall (x, z) \in \mathbb R^2, n\in\mathbb{Z}.
\]
We assume that $\varepsilon$ and $\mu$ are constants away from the region
$\Omega=\{(x, z): 0 \leq x \leq \Lambda,\, h_2 \leq  z \leq h_1\}$, where $h_j$
are constants. Denote $\Omega_1:=\{(x, z): 0 \leq x \leq \Lambda,\, z >
h_1\}$ and $\Omega_2:=\{(x, z): 0 \leq x \leq \Lambda,\, z < h_2\}$. There exist
constants $\varepsilon_j$ and $\mu_j$ such that
\[
\varepsilon (x, z)= \varepsilon_j,\quad \mu(x, z) =\mu_j\quad  \text{in}~
\Omega_j.
\]
Throughout we also assume that $\varepsilon \mu \geq \varepsilon_1 \mu_1$,
which is usually satisfied since $\varepsilon_1$ and $\mu_1$ are the electric
permittivity and magnetic permeability in the free space $\Omega_1$. Finally we
define $\Gamma_1=\{(x, z): 0 \leq x \leq \Lambda,\, z=h_1\}$ and $\Gamma_2:
=\{(x, z): 0 \leq x \leq \Lambda,\, z=h_2\}$. 

\begin{figure}
\centering
\includegraphics[width=0.3\textwidth]{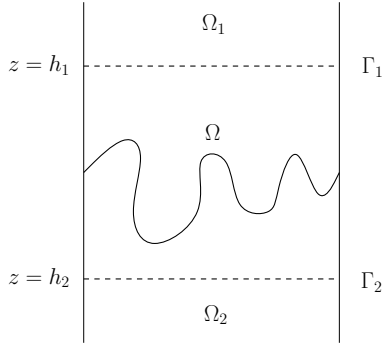}
\caption{Problem geometry of the time-domain scattering by a periodic
structure}
\label{pg}
\end{figure}

Consider an incoming plane wave $E^{\rm inc}$ which is incident on the structure
from above. Explicitly we have
\[
 E^{\rm inc} (x, z, t)= f(t- c_1 x - c_2 z),
\]
where $f$ is a smooth function and its regularity will be specified
later, and $c_1=\cos\theta/c, c_2=\sin\theta/c$. Here
$\theta$, satisfying $0<\theta < \pi$, is the incident angle,
and $c=1/\sqrt{\varepsilon_1\mu_1}>0$ is the light speed in the free
space. Clearly,  the incident field $E^{\rm inc} (x, z, t)$ satisfies
the wave equation (\ref{TE}) when $\varepsilon=\varepsilon_1, \mu=\mu_1$. 

Although the incident field $E^{\rm inc}$ may not be a periodic function in the
$x$-direction, we can verify that 
\[
 E^{\rm inc}(x+\Lambda, z, t) =E^{\rm inc} (x, z, t-c_1 \Lambda),\quad\forall
(x, z)\in\mathbb{R}^2, ~ t>0.
\]
Motivated by the uniqueness of the solution, we assume that the  total field
satisfies the same translation property, i.e., 
 \[
 E(x+\Lambda, z, t)=E(x, z, t- c_1\Lambda),\quad  (x, z)\in \mathbb R^2, ~ t>0.
 \]
We define
\begin{equation}\label{CV}
U(x, z, t)= E(x, z, t+c_1 (x-\Lambda)),\quad U^{\rm inc}(x, z, t)= E^{\rm
inc}(x, z, t+c_1(x- \Lambda)).
\end{equation}
It follows from (\ref{CV}) that we get
\[
U(x+\Lambda, z, t)=E(x+\Lambda, z, t+c_1 x)=E(x, z, t+c_1
x-c_1 \Lambda)=U(x, z, t),
\]
which shows that $U$ is a periodic function in the $x$-direction with period
$\Lambda$. Similarly, we can verify that the incident field $U^{\rm inc}$ is a
trivially periodic function of $x$ (independent of $x$) since  
\[
U^{\rm inc} (x, z, t)= E^{\rm inc} (x, z, t+c_1(x- \Lambda))=f(t-c_2
z-c_1 \Lambda).
\]

Using the change of variables, we have
\[
\partial_t E= \partial_t U,  \quad \partial_x E =\partial_x U
-c_1\partial_t U.
\]
The equation (\ref{TE}) becomes
\begin{equation}\label{TEU}
(\varepsilon -c_1^2 \mu^{-1})\partial_t^2 U=\nabla \cdot
(\mu^{-1} \nabla U) -c_1(\mu^{-1} \partial_{tx} U
+\partial_x(\mu^{-1} \partial_t U)).
\end{equation}
A simple calculation yields that 
\begin{align*}
 \varepsilon -c_1^2
\mu^{-1}&=(\varepsilon\mu-\varepsilon_1\mu_1\cos^2\theta)\mu^{-1}\geq
\varepsilon_1\mu_1(1-\cos^2\theta)\mu^{-1}\\
&=\varepsilon_1\mu_1\mu^{-1}\sin^2\theta>0,\quad\forall~\theta\in(0,\,\pi),
\end{align*}
which shows that the equation \eqref{TEU} is a well-defined wave equation. 

It is easy to verify that the incident field $U^{\rm inc}$ satisfies \eqref{TEU}
with $\varepsilon= \varepsilon_1, \mu=\mu_1$. To impose the initial conditions,
we assume that the total field and the incident field vanish for $t<0$ so that
the incident field $U^{\rm inc} =0$ and the scattered field $V =U-U^{\rm inc
}=0$ for $t<0$. The initial conditions are 
\begin{equation}\label{IC}
U|_{t=0}= \partial_t U|_{t=0} =0. 
\end{equation}
In addition $U$ is $\Lambda$-periodic in the $x$-direction. This paper aims to
study the well-posedness and stability of the scattering problem
\eqref{TEU}--\eqref{IC}. 

We introduce some notation. For any $s=s_1+{\rm i}s_2$ with $s_1,
s_2\in\mathbb{R}, s_1>0$, define by $\breve{u}(s)$ the Laplace transform of the
function $u(t)$, i.e., 
\[
 \breve{u}(s)=\mathscr{L}(u)(s)=\int_0^\infty e^{-st}u(t){\rm d}t.
\]
Define a weighted periodic function space
\[
H_{s, \rm p}^1(\Omega)=\{ u \in H^1(\Omega): u(0, z)=u(\Lambda, z)\},
\]
which is Sobolev space with the norm characterized by 
\[
\|u\|^2_{H_{s, \rm p}^1(\Omega)}=\int_\Omega \bigl( |\nabla
u|^2+|s|^2|u|^2\bigr) {\rm d} x {\rm d}z.
\]
Given $u \in H_{s, \rm p}^1(\Omega),$ it has a Fourier expansion with respect to
$x$:
\[
 u(x, z)=\sum_{n\in\mathbb{Z}}u_n(z) e^{{\rm i}\alpha_n x},\quad
\alpha_n=2n\pi\Lambda^{-1}. 
\]
A simple calculation yields an equivalent norm in $H_{s, \rm p}^1(\Omega)$
via Fourier coefficients:
\begin{equation}\label{HPN}
\|u\|^2_{H^1_{s, \rm p}(\Omega)}=\sum\limits_{n \in \mathbb Z}
\bigl(|s|^2+\alpha_n^2 \bigr) \int_{h_2}^{h_1} |u_n(z)|^2 {\rm d}z
+\sum\limits_{n \in \mathbb Z}\int_{h_2}^{h_1} |u_n'(z)|^2 {\rm d}z.
\end{equation}
For a periodic function $u$ defined on $\Gamma_j$ with Fourier coefficients
$u_n$, we define a weighted trace functional space
\begin{equation}\label{HTN}
H_s^\lambda (\Gamma_j)=\{ u \in L^2(\Gamma_j):
\|u\|^2_{H^\lambda(\Gamma_j)}=\sum\limits_{n \in \mathbb
Z}\bigl(|s|^2+\alpha_n^2\bigr)^\lambda |u_n|^2 < \infty\},
\end{equation}
where $\lambda\in\mathbb{R}$. It is clear to note that the dual space of
$H^{1/2}_s(\Gamma_j)$ is $H^{-1/2}_s(\Gamma_j)$ under the $L^2(\Gamma_j)$ inner
product
\[
\langle u, v \rangle_{\Gamma_j}=\int_{\Gamma_j} u \bar{v} {\rm d} \gamma_j.
\]

The weighted Sobolev spaces $H^1_{s, \rm p}(\Omega)$ and $H^\nu_s(\Gamma_j)$
are equivalent to the standard Sobolev spaces $H^1_{\rm p}(\Omega)$ and
$H^\lambda(\Gamma_j)$ since $|s|\neq 0$. Hereafter, the expression $` a \lesssim
b"$ stands for $``a \leq  C b "$, where $C$ is a positive constant and its
specific value is not required but should be always clear from the context.

\subsection{Transparent boundary condition}

We introduce a TBC to reformulate the scattering problem into an equivalent
initial-boundary value problem in a bounded domain. The idea is to design a
Dirichlet-to-Neumann (DtN) operator which maps the Dirichlet data to the
Neumann data of the wave field.

Subtracting the incident field $U^{\rm inc}$ from the total field
$U$ in \eqref{TEU} and \eqref{IC}, we obtain the equation for the scattered
field 
\begin{equation}\label{SE}
(\varepsilon_1 -c_1^2\mu_1^{-1})\partial_t^2 V=\nabla \cdot
(\mu^{-1}_1 \nabla V)-c_1(\mu^{-1}_1 \partial_{t x} V+\partial_x
(\mu^{-1}_1 \partial_t V)) \quad \text{in}~\Omega_1, ~ t>0,
\end{equation}
and the initial conditions
\begin{equation}\label{ICV}
 V|_{t=0}=\partial_t V|_{t=0}=0 \quad\text{in}~ \Omega_1.
\end{equation}

Let $\breve V(x, z, s) =\mathscr L (V)$  be the Laplace transforms of $V (x, z,
t)$ with respect to $t$. Recall that
\begin{align*}
\mathscr L (\partial_t V )&= s \breve V(x, z, s) -V(x, z, 0),\\
\mathscr L (\partial_t^2 V )&= s^2 \breve V(x, z, s) -s V(x, z, 0)
-\partial_t V(x, z, 0).
\end{align*}
Taking the Laplace transform of $(\ref{SE})$ and using the initial conditions
(\ref{ICV}), we have 
\[
(\varepsilon_1 -c_1^2\mu_1^{-1}) s^2  \breve {V}=\nabla \cdot(\mu
^{-1}_1 \nabla \breve{V})-c_1(\mu^{-1}_1 s\partial_x \breve V+ s
\partial_x(\mu^{-1}_1 \breve{V} ) ),
\]
which reduces to 
 \begin{equation}\label{LSE1}
 (\varepsilon_1 \mu_1 -c_1^2) s^2 \breve {V}= \Delta \breve
{V} - 2c_1 s \partial_x \breve V \quad \text{in}~\Omega_1.
 \end{equation}
Since $\breve V$ is a periodic function in $x$, it has the Fourier expansion
\[
\breve{V} (x, z) =\sum \limits_{n \in \mathbb Z} \breve {V}_n (z) e^{{\rm i}
\alpha_n x}, \quad z> h_1.
\]
Substituting the Fourier expansion of $\breve{V}$ into (\ref{LSE1}), we obtain
an ordinary differential equation for the Fourier coefficients:
\[
 \begin{cases}
  \partial^2_z \breve V_n (z)- (\beta_1^{(n)})^2 \breve V_n (z)=0,\quad z>h_1,\\
 \breve V_{n} (z) =  \breve V_{n} (h_1)
 \end{cases}
\]
where 
\[
 \beta_1^{(n)}=(\varepsilon_1 \mu_1 s^2 + (\alpha_n +{\rm i}c_1
s)^2)^{1/2},\quad {\rm Re}\beta_1^{(n)}<0. 
\]
Using the outgoing radiation condition, we have
\[
\breve V_n (z)=\breve V_n (h_1) e^{\beta_1^{(n)}(z-h_1)},
\]
Thus we get the Rayleigh expansion for the scattered field in $\Omega_1$:
\[
\breve V  (x, z)= \sum\limits_{n \in \mathbb Z} \breve V_{n}(h_1) e^{{\rm i}
\alpha_n x } e^{\beta_1^{(n)}(z-h_1)}. 
\]
Taking the normal derivative of the above equation on $\Gamma_1$ yields
\[
\partial_{\nu_1} \breve V (x, h_1)=\sum \limits_{n\in \mathbb Z}
\beta_1^{(n)} \breve V_n (h_1) e^{{\rm i} \alpha_n x},
\]
where $\nu_1=[0,\,1]^\top$ is the unit normal vector on $\Gamma_1$. 

Similarly, we can obtain the Rayleigh expansion for the total field in
$\Omega_2$: 
\[
\breve U (x, z)= \sum\limits_{n \in \mathbb Z} \breve U_{n}(h_2) e^{{\rm i}
\alpha_n x } e^{-\beta_2^{(n)}(z-h_2)},
\]
where 
\[
 \beta_2^{(n)}=(\varepsilon_2 \mu_2 s^2 +(\alpha_n +{\rm i}c_1 s)^2)^{1/2},\quad
{\rm Re}\beta_2^{(n)}<0. 
\]
Taking the normal derivative of $\breve{U}$ on $\Gamma_2$ gives 
\[
\partial_{\nu_2} \breve U(x, h_2)=\sum \limits_{n\in
\mathbb Z}\beta_2^{(n)} \breve U_n(h_2) e^{{\rm i} \alpha_n x},
\]
where $\nu_2=[0,\,-1]^\top$ is the normal vector on $\Gamma_2$. For any function
$u (x, h_j)$ defined on $\Gamma_j$, we define the DtN operators 
\begin{equation}\label{DTM}
 (\mathscr B_j u)(x, h_j)= \sum \limits_{n \in \mathbb Z}\beta_j^{(n)}
u_n (h_j) e^{{\rm i} \alpha_n x},\quad u(x, h_j)=\sum_{n\in\mathbb
Z}u_n(h_j)e^{{\rm i}\alpha_n x}. 
\end{equation}

\begin{lemm}\label{TT}
There exists a positive constant $C_1$ such that
\[
\|u\|_{H_s^{1/2}(\Gamma_j)} \leq C_1 \|u\|_{H_{s, \rm
p}^{1}(\Omega)},\quad\forall ~ u \in H_{s, \rm p}^1(\Omega). 
\]
\end{lemm}

\begin{proof}
First we have
\begin{align*}
(h_1 -h_2)|\zeta (h_j)|^2
&=\int_{h_2}^{h_1}|\zeta (z)|^2 {\rm d}z + \int_{h_2}^{h_1} \int_{z}^{h_j}
\frac{\rm d}{{\rm d} t} |\zeta (t)|^2 {\rm d}t {\rm d}z\\
&\leq \int_{h_2}^{h_1} |\zeta (z)|^2 {\rm d}z+(h_1 -h_2) \int_{h_2}^{h_1} 2
|\zeta (z)| |\zeta '(z)| {\rm d}z,
\end{align*}
which gives
\begin{align*}
\bigl(|s|^2+\alpha_n^2\bigr)^{1/2} |\zeta (h_j)|^2
\leq & (h_1-h_2)^{-1}\bigl(|s|^2+\alpha_n^2\bigr)^{1/2}
\int_{h_2}^{h_1} |\zeta (z)|^2 {\rm d}z\\
&+\int_{h_2}^{h_1} 2\bigl(|s|^2+\alpha_n^2\bigr)^{1/2} |\zeta (z)|
|\zeta '(z)| {\rm d}z. 
\end{align*}
It follows from the Cauchy--Schwarz inequality that
\begin{align*}
\bigl(|s|^2+\alpha_n^2\bigr)^{1/2} |\zeta (h_j)|^2
\leq & (h_1-h_2)^{-1}\bigl(|s| +|\alpha_n|\bigr)\int_{h_2}^{h_1}
|\zeta (z)|^2 {\rm d}z\\
&+\bigl(|s|^2+\alpha_n^2\bigr) \int_{h_2}^{h_1} |\zeta (z)|^2 {\rm
d}z+\int_{h_2}^{h_1} |\zeta '(z)|^2 {\rm d}z.
\end{align*}
Using the fact that $s=s_1 +{\rm i } s_2$ with $s_1>0$, we have 
\[
|s|\leq  s_1^{-1}|s|^2
,\quad |\alpha_n|\leq  (2\pi)^{-1}\Lambda\alpha_n^2 .
\]
Letting
\[
C_1^2=\max\{1+(h_1-h_2)^{-1}s_1^{-1},~
1+(2\pi)^{-1}(h_1-h_2)^{-1}\Lambda\},
\]
we can show that
\[
\bigl(|s|^2+\alpha_n^2\bigr)^{1/2} |\zeta (h_j)|^2 \leq
C_1^2\Bigl( \bigl(|s|^2+\alpha_n^2\bigr)
\int_{h_2}^{h_1} |\zeta (z)|^2 {\rm d}z+\int_{h_2}^{h_1} |\zeta '(z)|^2 {\rm d}
z\Bigr).
\]
The proof is completed by combing the above estimates and the definition
(\ref{HPN}).
\end{proof}

\begin{lemm} \label{DTN}
The DtN operator $\mathscr B_j: H_{s, \rm p}^{1/2}(\Gamma_j) \to
H_{s, \rm p}^{-1/2}(\Gamma_j)$ is continuous, i.e.,
\[
\|\mathscr B_j u\|_{H_{s, \rm p}^{-1/2}(\Gamma_j)}\leq C_2
\|u\|_{H_{s, \rm p}^{1/2}(\Gamma_j)},
\]
where $C_2>0$ is a constant. 
\end{lemm}

\begin{proof}
For any $u \in H_{s, \rm p}^{1/2}(\Gamma_j)$, it follow form (\ref{HTN}) that
\begin{align*}
\| \mathscr B_j u\|^2_{H_{s, \rm p}^{-1/2}(\Omega)}
&=\sum\limits_{n \in \mathbb Z} \bigl(|s|^2+\alpha_n^2\bigr)^{-1/2}
|\beta^{(n)}_j|^2 |u_n(h_j)|^2\\
&=\sum\limits_{n \in \mathbb Z}
\bigl(|s|^2+\alpha_n^2\bigr)^{1/2}
\bigl(|s|^2+\alpha_n^2\bigr)^{-1}|\beta^{(n)}_j|^2|u_n(h_j)|^2\\
&\leq C^2_2 \|u\|^2_{H_{s, \rm p}^{1/2}(\Gamma_j)},
\end{align*}
where we have used 
\[
|\beta_j^{(n)}|^2=|\varepsilon_j \mu_j s^2 +(\alpha_n+{\rm i}c_1 s)^2|\leq
\varepsilon_j \mu_j |s|^2+2(\alpha_n^2+c_1^2|s|^2)\leq C_2^2
\bigl(|s|^2+\alpha_n^2\bigr).
\]
Here
\[
C_2^2=\max\{2,\, 2c_1^2+\varepsilon_{\rm max}\mu_{\rm max}\},
\]
which completes the proof. 
\end{proof}

\begin{lemm}\label{TP}
We have the estimate
\[
{\rm Re}\langle (s\mu_j)^{-1}\mathscr B_j u, 
u\rangle_{\Gamma_j}\leq 0,\quad\forall ~ u \in H_{s, \rm p}^{1/2}(\Gamma_j).
\]
\end{lemm}

\begin{proof}
It follows from the definitions of (\ref{DTM}) and  (\ref{HTN}) that we have
\[
\langle (s\mu_j)^{-1}\mathscr B_j u,  u\rangle_{\Gamma_j}=\sum
\limits_{n \in \mathbb Z} \frac{\bar s\beta_j^{(n)}}{|s|^2 \mu_j} 
|u_n(h_j)|^2.
\]
Let $\beta_j^{(n)} =a_j+{\rm i} b_j, s =s_1 + {\rm i}
s_2$  with $s_1> 0, a_j<0$. Taking the real part of the above equation gives
\begin{equation}\label{TP-6}
{\rm Re} \langle (s\mu_j)^{-1}\mathscr B_j u, 
u\rangle_{\Gamma_j}=\sum\limits_{n \in \mathbb Z}\frac{ (s_1 a_j + s_2
b_j)}{|s|^2\mu_j}|u_n(h_j)|^2.
\end{equation}
Recalling $(\beta_j^{(n)})^2= \varepsilon_j \mu_j s^2 + (\alpha_n+{\rm
i}c_1 s)^2$, we have
\begin{equation}\label{TP-1}
 a_j^2-b_j^2=(\varepsilon_j
\mu_j-c_1^2)(s_1^2-s_2^2)+\alpha_n^2-2\alpha_n c_1 s_2
\end{equation}
and
\begin{equation}\label{TP-2}
a_j b_j=(\varepsilon_j \mu_j-c_1^2)s_1 s_2+\alpha_n c_1
s_1.
\end{equation}
Using \eqref{TP-2}, we get
\begin{equation}\label{TP-3}
s_1 a_j + s_2
b_j=\frac{s_1}{a_j}\bigl[a_j^2+(\varepsilon_j\mu_j-c_1^2)s_2^2+\alpha_n
c_1 s_2\bigr].
\end{equation}
Plugging \eqref{TP-1} into \eqref{TP-3} gives
\begin{equation}\label{TP-4}
s_1 a_j + s_2 b_j=\frac{s_1}{a_j}\bigl[
b_j^2+(\varepsilon_j\mu_j-c_1^2)s_1^2+\alpha_n^2-\alpha_n c_1 s_2\bigr]. 
\end{equation}
Adding \eqref{TP-3} and \eqref{TP-4}, we obtain
\begin{equation}\label{TP-5}
s_1 a_j + s_2 b_j=\frac{s_1}{2a_j}\bigl[
a_j^2+b_j^2+(\varepsilon_j\mu_j-c_1^2)(s_1^2+s_2^2)+\alpha_n^2\bigr].
\end{equation}
Substituting \eqref{TP-5} into \eqref{TP-6} yields 
\begin{align*}
&{\rm Re} \langle (s\mu_j)^{-1}\mathscr B_j u, 
u\rangle_{\Gamma_j}\\
&=\sum\limits_{n \in \mathbb Z}\frac{s_1}
{2a_j |s|^2\mu_j} \bigl[ a_j^2 +b_j^2
+(\varepsilon_j \mu_j-c_1^2)(s_1^2+s_2^2) +\alpha_n^2\bigr]|u_n(h_j)|^2\leq 0,
\end{align*}
which completes the proof. 
\end{proof}

Using the DtN operators (\ref{DTM}), we obtain the following TBC in
the $s$-domain: 
\begin{equation}\label{TBC}
\begin{cases}
\partial_{\nu_1} \breve U =\mathscr B_1 \breve U +\breve\rho&\quad\text{on} ~
\Gamma_1,\\
\partial_{\nu_2} \breve U =\mathscr B_2 \breve U &\quad\text{on}~\Gamma_2,
 \end{cases}
\end{equation}
where  $\breve \rho =\partial_z \breve U^{\rm inc}-\mathscr B_1 \breve U^{\rm
inc}$. Taking the inverse Laplace transform of (\ref{TBC}) yields the TBC in the
time domain:
\begin{equation}\label{TTBC}
\begin{cases}
\partial_{\nu_1} U = \mathscr T_1 U +\rho \quad \text{on}~\Gamma_1,\\
\partial_{\nu_2} U = \mathscr T_2 U \quad\text{on} ~\Gamma_2,
\end{cases}
\end{equation}
where $\rho$ is the inverse Laplace transform of $\breve\rho$,
i.e., $\rho=\mathscr{L}^{-1}(\breve\rho)$, and $\mathscr T_j =\mathscr
L^{-1}\circ \mathscr B_j\circ \mathscr L.$

\section{The Reduced Problem}\label{SRP}

In this section, we present the main results of this work, which include
the well-posedness and stability of the scattering problem and related a
priori estimates.

\subsection{Well-posedness in the $s$-domain} 

Taking the Laplace transform of \eqref{TEU} and using the TBC (\ref{TBC}), we
may consider the following reduced boundary value problem:
\begin{equation}\label{RP}
\begin{cases}
(\varepsilon -c_1^2\mu^{-1}) s  \breve {U}=\nabla
\cdot((s\mu)^{-1} \nabla \breve{U})-c_1(\mu^{-1} \partial_x
\breve U+ \partial_x(\mu^{-1} \breve{U} ) )&\quad\text{in}~\Omega,\\
\partial_{\nu_1} \breve U=\mathscr B_1 \breve U+\breve
\rho&\quad\text{on}~\Gamma_1,\\
\partial_{\nu_2} \breve U=\mathscr B_2 \breve U &\quad\text{on}~\Gamma_2.
\end{cases}
\end{equation}
Next we introduce a variational formulation of the boundary value problem
(\ref{RP}) and give a proof of its well-posedness in the space $H^1_{s,
\rm p}(\Omega)$. 

Multiplying (\ref{RP}) by the complex conjugate of a test function $v \in
H_{s, \rm p}^1(\Omega)$, using the integration by parts and TBCs, we arrive at
the variational problem: To find $\breve
U\in H_{s, \rm p}^1(\Omega)$ such that 
\begin{equation}\label{UVF}
a(\breve U, v) =\langle(s\mu_1)^{-1} \breve \rho, v
\rangle_{\Gamma_1},\quad\forall ~ v\in H_{s, \rm p}^1(\Omega),
\end{equation}
where the sesquilinear form
\begin{align}\label{SLF}
a(\breve U, v) = &\int_\Omega \big[(s \mu)^{-1} \nabla \breve U  \cdot \nabla
{\bar v}+(\varepsilon -c_1^2\mu^{-1}) s \breve U \bar v
+c_1 (\mu^{-1} \partial_x \breve U +\partial_x(\mu^{-1} \breve U))
\bar v \big] {\rm d}x {\rm d}z\notag\\
&- \sum\limits_{j=1}^2 \langle(s\mu_j) ^{-1}  \mathscr B_j \breve U, v
\rangle_{\Gamma_j}.
\end{align}

\begin{theo}\label{VP}
The variational problem (\ref{UVF}) has a unique solution $\breve U \in
H_{s, \rm p}^1(\Omega)$, which satisfies 
\[
\| \nabla \breve U\|_{L^2(\Omega)^2} +\|s\breve U\|_{L^2(\Omega)}
\lesssim s_1^{-1} |s|\| \breve \rho\|_{H_s^{-1/2}(\Gamma_1)}.
\]
\end{theo}

\begin{proof}
It suffices to show the coercivity of the sesquilinear form of $a$, since the
continuity follows directly from the Cauchy--Schwarz inequality, Lemma
\ref{TT}, and Lemma \ref{DTN}. 

Letting $v=\breve U$ in \eqref{SLF}, we get 
\begin{align*}
a(\breve U, \breve U) = &\int_\Omega \big[(s \mu)^{-1} |\nabla \breve
U|^2 +(\varepsilon -c_1^2\mu^{-1}) s |\breve U|^2 +c_1 (\mu^{-1}
\partial_x \breve U +\partial_x(\mu^{-1} \breve U))
\bar{\breve U} \big] {\rm d}x {\rm d}z\notag\\
&- \sum\limits_{j=1}^2 \langle (s\mu_j) ^{-1} \mathscr B_j \breve U, \breve U
\rangle_{\Gamma_j}.
\end{align*}
Taking the real part of the above equation yields
\begin{align*}
{\rm Re}\, a (\breve U, \breve U)
=&\int_{\Omega} \bigl(\frac {s_1} {|s|^2 \mu} |\nabla \breve U|^2 +(\varepsilon
-c_1^2\mu^{-1}) s_1 |\breve  U|^2\bigr) {\rm d}x {\rm d}z-{\rm Re} \sum
\limits_{j=1}^2 \langle (s\mu_j) ^{-1} \mathscr B_j \breve U, \breve U
\rangle _{\Gamma_j} \\
&+  c_1 {\rm Re} \int _{\Omega}\bigl(\mu^{-1 } \partial_x \breve U 
\bar {\breve U} + \partial_x (\mu^{-1} \breve U) \bar {\breve U}\bigr){\rm d}x
{\rm d}z.
\end{align*}
Since $\mu $ and $\breve U$ are periodic in $x$, we have from the integration by
part that 
\[
\int _{\Omega}\bigl(\mu^{-1} \partial_x \breve U  \bar {\breve U} + \partial_x
(\mu^{-1} \breve U) \bar {\breve U}\bigr) {\rm d}x{\rm d}z+\int_\Omega
\bigl(\breve U \partial_x(\mu^{-1} \bar {\breve U})+\mu^{-1} \breve U \partial_x
\bar {\breve U}\bigr){\rm d}x {\rm d}z=0,
\]
which gives 
\[
 {\rm Re }\int _{\Omega}\bigl(\mu^{-1}\partial_x \breve U \bar {\breve U} +
\partial_x (\mu^{-1} \breve U) \bar {\breve U}\bigr){\rm d}x {\rm d}z=0.
\]
Combining the above estimate and Lemma \ref{TP}, we obtain 
\begin{equation}\label{CP}
{\rm Re}\, a (\breve U, \breve U) \geq C\frac {s_1}{ |s|^2}
\int_{\Omega } \bigl(|\nabla \breve U|^2 + |s\breve U|^2 \bigr) {\rm d} x
{\rm d}z,
\end{equation}
where  $C= \mu^{-1}_{\rm max}\min \{1,\, \varepsilon_1\mu_1\sin^2\theta\}$.

It follows from the Lax--Milgram lemma that the variational problem (\ref{UVF})
has a unique solution $\breve U \in H_{s, \rm p}^1(\Omega).$
Moreover, we have from (\ref{UVF}) and Lemma {\ref{TT}} that
\begin{equation}\label{AE}
|a (\breve U, \breve U)| \leq (|s|\mu_1)^{-1} \| \breve
\rho\|_{H_s^{-1/2}(\Gamma_1)} \| \breve U\|_{H_s^{1/2}(\Gamma_1)}\leq
C_1 (|s|\mu_1)^{-1} \| \breve \rho\|_{H_s^{-1/2}(\Gamma_1)} \| \breve
U\|_{H_{s, \rm p}^{1}(\Omega)}.
\end{equation}
Combing (\ref{CP}) and (\ref{AE}) leads to
\[
\|\nabla \breve U\|^2_{L^2(\Omega)^2} +\|s\breve
U\|^2_{L^2(\Omega)} \lesssim s_1^{-1} |s|\| \breve 
\rho\|_{H_s^{-1/2}(\Gamma_1)}\| \breve
U\|_{H_{s, \rm p}^{1}(\Omega)},
\]
which completes the proof after applying the Cauchy--Schwarz inequality.
\end{proof}

\subsection{Well-posedness in the time-domain}

Using the time-domain TBC \eqref{TTBC}, we consider the reduced
initial-boundary value problem:
\begin{equation}\label{TRP}
\begin{cases}
(\varepsilon -c_1^2\mu^{-1})\partial_t^2 U=\nabla \cdot (\mu^{-1}
\nabla U) -c_1(\mu^{-1} \partial_{tx} U +\partial_x(\mu^{-1}
\partial_t U)) &\quad\text{in}~\Omega, ~ t>0,\\
U|_{t=0} =\partial_t U|_{t=0}=0 &\quad\text{in} ~ \Omega, \\
\partial_{\nu_1} U= \mathscr T_1 U + \rho &\quad\text{on}~\Gamma_1, ~t>0,\\
\partial_{\nu_2} U= \mathscr T_2 U  &\quad\text{on}~\Gamma_2, ~t>0.
\end{cases}
\end{equation}

The following lemma (cf. \cite[Theorem 43.1]{Treves 1975})  is an
analogue of Paley--Wiener--Schwarz theorem for Fourier transform of the
distributions with compact support in the case of Laplace transform.

\begin{lemm}\label {A2}
 Let $\breve h(s)$ denote a holomorphic function in the half-plane
$s_1 > \sigma_0$ , valued in the Banach space $\mathbb E$. The two following conditions are
equivalent:
\begin{enumerate}

\item there is a distribution $h \in \mathcal D_{+}'(\mathbb E)$ whose Laplace
transform is equal to $\breve h(s)$;

\item there is a real $\sigma_1$ with $\sigma_0 \leq \sigma_1 <\infty$ and an
integer $m \geq 0$ such that for all complex numbers $s$ with ${\rm Re} s =s_1 >
\sigma_1,$ it holds that $\| \breve h (s)\|_{\mathbb E} \lesssim (1+|s|)^{m}$,

\end{enumerate}
where $\mathcal D'_{+}(\mathbb E)$ is the space of distributions on the real
line which vanish identically in the open negative half line.
\end{lemm}

\begin{theo}
The initial-boundary value problem  (\ref{TRP}) has a unique solution $U(x, z,
t)$, which satisfies
\[
U(x, z, t) \in L^2(0, T; H_{\rm p}^1(\Omega)) \cap H^{1}(0, T; L^2(\Omega))
\]
and the stability estimate
\begin{align}\label{ST}
\max \limits_{t \in [t, T]}
&\big(\|\partial_t U\|_{L^2(\Omega)}+\|\partial_t (\nabla U)\|_{L^2 (\Omega)^2} 
\big)  \lesssim \big( \|\rho\|_{L^1(0, T; H^{-1/2}(\Gamma_1))}\nonumber\\
&+\max \limits_{t \in [0, T ]}\|\partial_t\rho\|_{
H^{-1/2}(\Gamma_1)}+\|\partial^2_t\rho\|_{L^1(0, T;
H^{-1/2}(\Gamma_1))}\big).
\end{align}
\end{theo}

\begin{proof}
First we have 
\begin{align*}
&\int_0^T \big( \|\nabla U\|_{L^2(\Omega)^2}^2+\|\partial_t
U\|^2_{L^2(\Omega)}\big) {\rm d} t\\
\leq & \int_0^T e^{- 2 s_1(t- T)}\big ( \|\nabla U\|_{L^2(\Omega)^2}^2 
+\|\partial_t U\|^2_{L^2(\Omega)}\big) {\rm d} t\\
=& e^{2 s_1 T} \int_0^T e^{-2 s_1 t} \big ( \|\nabla U\|_{L^2(\Omega)^2}^2 
+\|\partial_t U\|^2_{L^2(\Omega)}\big) {\rm d} t\\
\lesssim & \int_0^{\infty} e^{-2 s_1 t}  \big ( \|\nabla U\|_{L^2(\Omega)^2}^2 
+\|\partial_t U\|^2_{L^2(\Omega)}\big) {\rm d} t.
\end{align*}
Hence it suffices to estimate the integral
\[
\int_0^{\infty} e^{-2 s_1 t}  \big ( \|\nabla U\|_{L^2(\Omega)^2}^2 
+\|\partial_t U\|^2_{L^2(\Omega)}\big) {\rm d} t.
\]
Taking the Laplace transform of (\ref{TRP}) yields
\begin{equation*}
\left\{
\begin{array}{ll}
(\varepsilon -c_1^2\mu^{-1}) s  \breve {U}=\nabla
\cdot((s\mu)^{-1} \nabla \breve{U})-c_1(\mu^{-1} \partial_x
\breve U+ \partial_x(\mu^{-1} \breve{U} ) )&\quad\text{in}~\Omega,\\
\partial_{\nu_1} \breve U=\mathscr B_1 \breve U+\breve \rho
&\quad\text{on}~\Gamma_1,\\
\partial_{\nu_2} \breve U=\mathscr B_2 \breve U &\quad\text{on}~\Gamma_2.
\end{array}
\right.
\end{equation*}
 The well-posedness of $\breve U \in H_{s, \rm p}^1(\Omega)$ follows directly
from Theorem {\ref{VP}}. By the trace theorem in Lemma \ref{TT}, we get
\[
\|\nabla  \breve U\|^2_{L^2 (\Omega)^2}+ \|s \breve U \|^2_{L^2(\Omega)}
\lesssim s_1^{-2} |s|^2  \|\breve \rho \|^2_{H^{-1/2}(\Gamma_1)} \lesssim
s_1^{-2} |s|^2 \|  \breve U^{\rm inc} \|^2_{H_{\rm p}^1(\Omega)}.
\]
It follows from \cite[Lemma 44.1]{Treves1975} that $\breve U$ is a holomorphic
function of $s$ on the half plane $s_1 >\bar\gamma>0,$  where $\bar \gamma$ is
any positive constant. Hence we have from Lemma \ref{A2} that the inverse
Laplace transform of $\breve U$ exists and is supported in $[0, \infty].$

One may verify from the inverse Laplace transform that
\[
\breve U =\mathscr L (U) =\mathscr F (e^{-s_1 t}  U),
\]
where $\mathscr F $ is the Fourier transform with respect to $s_2$. Recall  the
Plancherel or Parseval identity for the Laplace transform (cf.
\cite[(2.46)]{Cohen2007})
 \begin{equation}\label{PI}
 \frac{1}{2 \pi} \int_{- \infty}^{\infty} \breve u(s)  \breve  v(s) {\rm d} s_2= \int_0^{\infty} e^{- 2 s_1 t}  u(t)
 v(t) {\rm d} t,\quad\forall ~ s_1>\lambda,
 \end{equation}
where $\breve u= \mathscr L (u), \breve  v= \mathscr L (v)$ and $\lambda$ is
abscissa of convergence for the Laplace transform of $u$ and $v.$

Using (\ref{PI}), we have
 \begin{align*}
 \int_0^{\infty} e^{-2s_1 t} \big(\|\nabla U\|^2_{L^2 (\Omega)^2}+\|\partial_t U
\|^2_{L^2(\Omega)}\big){\rm d}t
 =\frac{1}{2 \pi} \int_{-\infty}^{\infty}\big( \|\nabla \breve
U\|^2_{L^2(\Omega)^2}+ \|s \breve U\|^2_{L^2(\Omega)}\big) {\rm d} s_2\\
 \lesssim  s_1 ^{-2} \int_{-\infty}^{\infty} |s|^2 \big( \|  \breve U^{\rm inc}
\|^2_{L^2(\Omega)} +\|\nabla\breve U^{\rm inc}\|^2_{L^2(\Omega)^2}\big) {\rm d}
s_2.
 \end{align*}
Since $U^{\rm inc}|_{t=0} =\partial_t U^{\rm inc}|_{t=0} =0$ in $\Omega$,  we
have $\mathscr L (\partial_t U^{\rm inc}) = s \breve U^{\rm inc}$ in $\Omega.$
It is easy to note that
\begin{align*}
|s|^2 \breve U^{\rm inc}=(2 s_1 -s) s \breve U^{\rm inc}& = 2 s_1 \mathscr
L(\partial_t U^{\rm inc}) -\mathscr L (\partial^{2}_t U^{\rm
inc}),\\
|s|^2 \nabla\breve U^{\rm inc}=(2 s_1 -s) s \nabla\breve U^{\rm inc} &= 2 s_1
\mathscr L(\partial_t \nabla U^{\rm inc}) -\mathscr L (\partial^{2}_t\nabla
U^{\rm inc}).
\end{align*}
Hence we have
 \begin{align*}
& \int_0^{\infty} e^{-2s_1 t} \big(\|\nabla U\|^2_{L^2
(\Omega)^2}+\|\partial_t U \|^2_{L^2(\Omega)}\big) {\rm d}t\\
&\lesssim \int_{- \infty}^ {\infty} \|\mathscr L (\partial_t U^{\rm
inc})\|^2_{L^2(\Omega)} {\rm d}s_2+  s_1^{-2}\int_{-\infty}^{\infty} \| \mathscr
L (\partial_t^2 U^{\rm inc})\|^2_{L^2 (\Omega)} {\rm d} s_2\\
&+\int_{- \infty}^ {\infty} \|\mathscr L (\partial_t \nabla U^{\rm
inc})\|^2_{L^2(\Omega)^2} {\rm d}s_2+  s_1^{-2}\int_{-\infty}^{\infty} \|
\mathscr L(\partial_t^2 \nabla U^{\rm inc})\|^2_{L^2(\Omega)^2} {\rm d} s_2.
 \end{align*}
 Using the Parseval identity \eqref{PI} again gives
 \begin{align*}
& \int_0^{\infty} e^{-2s_1 t} \big(\|\nabla U\|^2_{L^2 (\Omega)^2}+\|\partial_t
U \|^2_{L^2(\Omega)}\big){\rm d}t\\
 &\lesssim\int_0^{\infty} e^{-2 s_1 t} \|\partial_t U^{\rm
inc}\|^2_{H^1_{\rm p}(\Omega)} {\rm d} t+ s_1^{-2} \int_0^{\infty} e^{- 2 s_1 t}
\|\partial_t^2 U^{\rm inc}\|^2_{H^1_{\rm p}(\Omega)} {\rm d} t,
 \end{align*}
 which shows that
 \[
 U(x, z, t) \in L^2(0, T; H_{\rm p}^1(\Omega) \cap H^{1}(0, T; L^2(\Omega)).
 \]

Next we prove the stability. Let $\tilde U (x, z, t)$ be the extension of $U(x,
z, t)$ with respect to $t$ in $\mathbb R$ such that $\tilde U (x, z, t)= 0$
outside the interval $[0, t].$ By the Parseval identity (\ref{PI}), we follow
the proof of Lemma \ref{TP} and get
 \begin{align*}
 {\rm Re} \int_0^t e^{-2 s_1 t} \langle\mathscr T_j U,  
\partial_t U\rangle_{\Gamma_j} {\rm d} t
 &= {\rm Re} \int_{\Gamma_j} \int_0^{\infty}
 e^{-2 s_1 t} \langle\mathscr T_j \tilde {U},  \partial_t \tilde
U\rangle_{\Gamma_j} {\rm d} t\\
 &=\frac{1}{2 \pi} \int_{-\infty}^{\infty} {\rm Re} \langle \mathscr T_j \breve
{\tilde U}, s \breve {\tilde U}\rangle_{\Gamma_j} {\rm d} s_2  \leq 0,
 \end{align*}
 which yields after taking $s_1 \rightarrow 0$ that
 \begin{equation}\label{TPT}
 {\rm Re} \int_0^t \langle \mathscr T_j U,  \partial_t U\rangle_{\Gamma_j}{\rm
d} t \leq 0.
 \end{equation}
 
 For any $0<t< T,$ consider the energy function
 \[
 e_1(t) =\|(\varepsilon -c_1^2\mu^{-1})^{1/2} \partial_t U(\cdot,
t)\|^2_{L^2 (\Omega)} +\|\mu^{-1/2} \nabla U(\cdot, t)\|^2_{L^2(\Omega)^2}.
 \]
It follows from (\ref{TRP}) that we have
 \begin{align*}
 \int_0^t e'(t) dt
 =&2 {\rm Re } \int_0^t \int_{\Omega}\big( (\varepsilon -
c_1^2\mu^{-1})\partial_t^2 U  \partial_t\bar {U}
 +\mu^{-1} \partial_t(\nabla U)\cdot \nabla \bar{U} \big) {\rm d} x {\rm
d} z {\rm d}t\\
 =& 2 {\rm Re} \int_0 ^t \int_\Omega \big( \nabla \cdot(\mu^{-1} \nabla U)
\partial_t \bar{U} +\mu^{-1} \partial_t(\nabla U) \cdot \nabla \bar{U}\big) {\rm
d} x {\rm d} z {\rm d}t\\
 &-2 {\rm Re} \int_0^t \int_\Omega \big(c_1(\mu^{-1} \partial_{tx} U+\partial_x
(\mu^{-1} \partial_t U))  \partial_t \bar{U}\big) {\rm d} x {\rm d} z {\rm
d}t.
 \end{align*}
 Since $\mu$ and $U$ are periodic functions in $x$, integrating by parts yields
 \[
 \int_0^t \int_{\Omega}\big( \mu^{-1} \partial_{tx}U  \partial_t
\bar{U}+\mu^{-1} \partial_{tx}\bar{U}  \partial_t  U+
 \partial_x(\mu^{-1} \partial_t U)  \partial_t \bar{U} +\partial_x(\mu^{-1}
\partial_t \bar{U}) \partial_t U \big) {\rm d} x {\rm d} z {\rm d}t=0,
 \]
which gives 
 \[
{\rm Re} \int_0^t \int_\Omega \big(c_1 (\mu^{-1} \partial_{tx}
U+\partial_x (\mu^{-1} \partial_t U))  \partial_t \bar{U}\big) {\rm d} x
{\rm d} z {\rm d}t=0.
 \]
 Since $e_1(0) =0$, we obtain from \eqref{TPT} that
 \begin{align*}
  e_1(t)=&\int_0^t e'(t) {\rm d} t
  = 2{\rm Re} \int_0^t  \int_\Omega \big(-\mu^{-1} \nabla U \cdot \partial_t 
(\nabla \bar{U})+\mu^{-1} \partial_t(\nabla U) \cdot \nabla \bar{U}\big) {\rm d}
x {\rm d} z {\rm d}t \\
  &+   2{\rm Re }\int_0^t \sum\limits_{j=1}^2 \int_{\Gamma_j} \mu_j^{-1}
\partial_\nu U   \partial_t \bar{U} {\rm d} \gamma_j {\rm d}t\\
 =& 2{\rm Re }\int_0^t  \sum\limits_{j=1}^2 \mu_j^{-1}\langle \mathscr
T_j U, \partial_t U\rangle_{\Gamma_j}{\rm d}t
 +2 {\rm Re}  \int_0^t \langle\rho, \partial_t U\rangle_{\Gamma_1} {\rm  d}t\\
 \leq &2{\rm Re}  \int_0^t \big(\|\rho\|_{H^{-1/2}(\Gamma_1)} \|\partial_t
U\|_{H^{1/2}(\Gamma_1)} \big){\rm d} t\\
 \lesssim  &2{\rm Re}  \int_0^t \big(\|\rho\|_{H^{-1/2}(\Gamma_1)} \|\partial_t
U\|_{H_{\rm p}^{1}(\Omega)}\big) {\rm  d} t\\
 \leq & 2\big( \max\limits_{t \in [0, T]}\|\partial_t U\|_{H^1_{\rm
p}(\Omega)} \big)\|\rho\|_{L^1(0, T; H^{-1/2}(\Gamma_1))}.
 \end{align*}
 Taking the derivative of  (\ref{TRP}) with respect to $t,$  we know that
$\partial_t U$ also satisfies the same equations with $\rho$ replaced by
$\partial_t \rho$. Hence, we may consider the similar energy function
 \[
 e_2(t) =\|(\varepsilon -c_1^2\mu^{-1})^{1/2} \partial^2_t U(\cdot,
t)\|^2_{L^2 (\Omega)} +\|\mu^{-1/2} \partial_t ( \nabla U(\cdot,
t))\|^2_{L^2(\Omega)^2}
 \]
and get the estimate
 \begin{align*}
 e_2 (t)
 \leq& 2 {\rm Re}  \int_0^t  \int_{\Gamma_1}\partial_t \rho \, \partial^2_t
\bar{U} {\rm d} \gamma_1 {\rm d }t\\
 =&2 {\rm Re}  \int_{\Gamma_1} \partial_t \rho \, \partial_t \bar U \mid_0^t
{\rm d} \gamma_1- 2 {\rm Re} \int_0^t \int_{\Gamma_1}\partial_t^2 \rho \,
\partial_t \bar U {\rm d} \gamma_1 {\rm d} t\\
 \leq& 2 \big(\max\limits_{t \in [0, T]}\|\partial_t U\|_{H^1_{\rm p}
(\Omega)}\big) \big(\max \limits_{t \in [0, T ]}\|\partial_t\rho\|_{
H^{-1/2}(\Gamma_1)}+\|\partial^2_t\rho\|_{L^1(0, T;
H^{-1/2}(\Gamma_1))}\big).
 \end{align*}
 Combing the above estimates, we can obtain
 \begin{align*}
 &\max\limits_{t \in [0, T]}\|\partial_t U\|^2_{H^1_{\rm p}
(\Omega)}\lesssim \max\limits_{t\in [0, T]}  e_1(t)+e_2(t)\\
&\lesssim \big( \|\rho\|_{L^1(0, T; H^{-1/2}(\Gamma_1))}+\max \limits_{t \in
[0, T ]}\|\partial_t\rho\|_{ H^{-1/2}(\Gamma_1)}+\|\partial^2_t\rho\|_{L^1(0,
T; H^{-1/2}(\Gamma_1))}\big) \| \partial_t U\|_{H_p^1(\Omega)},
 \end{align*}
 which give the estimate (\ref{ST}) after applying the Cauchy--Schwarz
inequality.
 \end{proof}
 
\subsection{A priori estimates}

In this section, we derive a priori estimates for the total field 
with a minimum regularity requirement for the data and an explicit dependence on
the time.

The variation problem of (\ref{TRP}) in time domain is to find $U \in
H^1_{\rm p}(\Omega)$ for all $t>0$ such that
\begin{align}\label{TVP}
\int_{\Omega}(\varepsilon-c_1^2\mu^{-1})\partial_t^2 U  \bar w {\rm
d} x {\rm d}z= -\int_{\Omega} \mu^{-1} \nabla U \cdot \nabla \bar w {\rm d} x 
{\rm d} z+\sum\limits_{j=1}^2 \int_{\Gamma_j} \mu_j^{-1} \mathscr T_j
U  \bar w {\rm d}\gamma_j \nonumber\\
+\int_{\Gamma_1} \rho  \bar w {\rm d} \gamma_1 -c_1\int_{\Omega}(\mu^{-1}
\partial_{tx} U+ \partial_x (\mu^{-1} \partial_t U)) \bar w {\rm d}x {\rm
d}z,\quad\forall~ w\in H^1_{\rm p}(\Omega). 
\end{align}
To show the stability of its solution, we follow the argument in \cite{Treves
1975} but with a careful study of the TBC.

\begin{theo}
Let $U \in H^1_{\rm p}(\Omega)$ be the solution of (\ref{TRP}). Given $\rho \in
L^1(0, T; H^{-1/2}(\Gamma_1))$,  we have for any $T>0$ that 
\begin{equation}\label{ES1}
\|U\|_{L^{\infty} (0, T; L^2(\Omega))} +\|\nabla U\|_{L^{\infty}(0, T;
L^2(\Omega))} \lesssim T\| \rho\|_{L^1(0, T; H^{-1/2} (\Gamma_1))}
+ \| \partial_t\rho\|_{L^1(0, T; H^{-1/2} (\Gamma_1))}.
\end{equation}
and
\begin{equation}\label{ES2}
\|U\|_{L^2(0, T; L^2(\Omega))}+\| \nabla U\|_{L^2 (0, T; L^2 (\Omega))}
\lesssim T^{3/2} \| \rho\|_{L^1(0, T; H^{-1/2}(\Gamma_1))}
+T^{1/2} \| \partial_t \rho\|_{L^1(0, T; H^{-1/2}(\Gamma_1))}.
\end{equation}

\end{theo}

\begin{proof}
Let $0 <\xi <T$ and define an auxiliary function
\[
\psi_1(x, z, t)=\int_t^\xi U(x, z, \tau) {\rm d} \tau, \quad(x, z) \in \Omega,
~0 \leq t \leq \xi.
\]
It is clear that
\begin{equation}\label{AP}
\psi_1 (x, z, \xi)=0,\quad\partial_t \psi_1 (x, z, t)=-U(x, z, t).
\end{equation}
For any $\phi (x, z, t) \in L^2(0, \xi;  L^2(\Omega)),$  we have
\begin{equation}\label{AP1}
\int_0^\xi \phi(x, z, t)\bar{\psi}_1 (x, z, t) {\rm d} t=\int_0^{\xi}\big(
\int_0^t \phi_1(x, z, \tau) d \tau\big)  \bar {U}(x, z, t) {\rm d} t.
\end{equation}
Indeed, using integration by parts and (\ref{AP}), we have
\begin{align*}
&\int_0^{\xi}
\phi(x, z, t)  \bar {\psi}_1 (x, z, t) {\rm d} t =\int_0^{\xi } \big( \phi(x, z,
t)  \int_t^{\xi} \bar U(x, z, \tau) {\rm d} \tau\big)  {\rm d} t\\
&=\int_0^{\xi} \int_t^{\xi} \bar U (x, z, \tau) {\rm d} \tau   {\rm
d}\big(\int_0^t \phi(x, z, \varsigma) {\rm d} \varsigma \big)\\
&=\int_t^{\xi} \bar U(x, z, \tau) {\rm d} \tau \int_0^t\phi(x, z, \varsigma)
{\rm d} \varsigma \mid_0^{\xi}+\int_0^{\xi}\big( \int_0^t \phi(x, z, \varsigma)
{\rm d} \varsigma \big)  \bar U(x, z, t) {\rm d} t\\
&=\int_0^{\xi} \big(\int_0^t \phi(x, z, \tau) {\rm d} \tau \big) \bar U(x, z, t)
{\rm d}  t.
\end{align*}
Next, we take the test function $w=\psi_1$ in (\ref{TVP}) and get
\begin{align}\label{TVP1}
\int_{\Omega}(\varepsilon-c_1^2\mu^{-1})\partial_t^2 U  \bar \psi_1
{\rm d} x {\rm  d}z
= -\int_{\Omega} \mu^{-1} \nabla U \cdot \nabla \bar \psi_1 {\rm d} x {\rm d}  z
+\sum\limits_{j=1}^2 \int_{\Gamma_j} \mu_j^{-1} \mathscr T_j U  \bar \psi_1
{\rm d} \gamma_j  \nonumber\\
+\int_{\Gamma_1} \rho  \bar \psi_1 {\rm d}  \gamma_1 -c_1\int_{\Omega}(\mu^{-1}
\partial_{tx} U+\partial_x (\mu^{-1} \partial_t U)) \bar \psi_1  {\rm d} x {\rm
d} z.
\end{align}
It follows from (\ref{AP})  and the initial conditions in (\ref{TRP}) that
\begin{align*}
{\rm Re} \int_0^{\xi} \int_\Omega (\varepsilon-c_1^2\mu^{-1})\partial_t^2 U 
\bar \psi_1 {\rm d} x {\rm d} z {\rm d} t
&={\rm Re} \int_{\Omega} \int_0^{\xi} \big(\partial_t (
(\varepsilon-c_1^2\mu^{-1})\partial_t U  \bar \psi_1)
+(\varepsilon-c_1^2\mu^{-1})\partial_t U  \bar U \big) {\rm d} t
{\rm d} x {\rm d} z \\
&={\rm Re } \int_{\Omega}\big( (\varepsilon-c_1^2\mu^{-1})\partial_t U  \bar
\psi_1)\mid_0^{\xi}+\frac{1}{2}(\varepsilon-c_1^2\mu^{-1})|U|^2
\mid_0^{\xi} \big){\rm d} x {\rm d} z  \\
&=\frac{1}{2}\|(\varepsilon -c_1^2\mu^{-1})^{1/2} U(\cdot,
\xi)\|^2_{L^2(\Omega)}.
\end{align*}
Integrating (\ref{TVP1}) from $t=0$ to $t= \xi$ and taking the real part yield
\begin{align}\label{EE1}
\frac{1}{2} &\| (\varepsilon -c_1^2\mu^{-1})^{1/2} U(\cdot,
\xi)\|^2_{L^2 (\Omega)}+{\rm Re}\int_0^{\xi} \int_{\Omega} \mu^{-1} \nabla U
\cdot \nabla \bar \psi_1 {\rm d} x {\rm d} z {\rm d} t\nonumber\\
=&\frac{1}{2}\| (\varepsilon -c_1^2\mu^{-1})^{1/2} U(\cdot,
\xi)\|^2_{L^2 (\Omega)}+\frac{1}{2}\int_{\Omega}\mu^{-1}|\int_0^{\xi} \nabla
U(\cdot, t) {\rm d}t|^2 {\rm d} x {\rm d} z \nonumber\\
=&{\rm Re} \int_0^{\xi} \sum\limits_{j=1}^2 \int_{\Gamma_j} \mu_j^{-1}
\mathscr T_j U  \bar\psi_1 {\rm d} \gamma_j {\rm d}t
+{\rm Re }\int_0^{\xi} \int_{\Gamma_1} \rho  \bar \psi_1 {\rm d} \gamma_1 {\rm d} t \nonumber \\
&-c_1 {\rm Re}\int_0^{\xi}  \int_{\Omega}(\mu^{-1} \partial_{tx} U+
\partial_x (\mu^{-1} \partial_t U)) \bar \psi_1 {\rm d} x {\rm d} z {\rm d} t.
\end{align}
In what follows, we estimate the three terms of the right-hand side of (\ref{EE1}) separately.

By the property (\ref{AP1}), we have
\begin{align*}
{\rm Re}\int_0^{\xi} \int_{\Gamma_j} \mu_j^{-1} \mathscr T_j U  \bar \psi_1 
{\rm d} \gamma_j {\rm d} t
={\rm Re} \int_0^{\xi} \int_0^{t}\big(\int_{\Gamma_j} \mu_j^{-1} \mathscr T_j
U(\cdot,\tau) {\rm d}\gamma_j \big) {\rm d}\tau  \bar U (\cdot, t) {\rm d}t.
\end{align*}
Let $\tilde U$ be the extension of $U$ with respect to $t$ in $\mathbb R$ such
that $\tilde U=0$ outside the interval $[0, \xi].$  We obtain from the Parseval
identity and  Lemma \ref {TP} that
\begin{align*}
{\rm Re}
 &\int_0^{\xi} e^{-2 s_1 t} \int_0^{t}\big(\int_{\Gamma_j} \mu_j^{-1} \mathscr
T_j U(\cdot,\tau) {\rm d}\gamma_j \big) {\rm d}\tau  \bar U (\cdot, t) {\rm d}
t\\
&= {\rm Re} \int_{\Gamma_j} \int_0^{\infty} e^{-2 s_1 t} \big( \int_0^t 
\mu_j^{-1} \mathscr T_j \tilde {U}(\cdot, \tau){\rm d\tau}\big)   \bar {\tilde
U}(\cdot, t) {\rm d} t d \gamma_j\\
&={\rm Re} \int_{\Gamma_j}\int_0^{\infty} e^{-2 s_1 t} \big(\int_0^t \mathscr
L^{-1} \circ \mu_j^{-1} \mathscr B_j \circ \mathscr L \tilde U(\cdot, \tau){\rm
d}\tau\big) \bar{\tilde {U}}(\cdot, t) {\rm d} \gamma_j  {\rm d} t\\
&={\rm Re} \int_{\Gamma_j}\int_0^{\infty} e^{-2 s_1 t} \big(\mathscr
L^{-1} \circ (s\mu_j)^{-1} \mathscr B_j \circ \mathscr L \tilde U(\cdot,
t)\big)\bar{\tilde {U}}(\cdot, t) {\rm d} \gamma_j  {\rm d} t\\
&=\frac{1}{2\pi}\int_{-\infty}^\infty {\rm Re}\langle (s\mu_j)^{-1}\mathscr
B_j\breve{ \tilde U}, \breve{\tilde U}\rangle_{\Gamma_j} {\rm d}s_2\leq 0,
\end{align*}
where we have used the fact that 
\[
\int_0^t u(\tau) {\rm d} \tau =\mathscr L^{-1}
(s^{-1} \breve u(s)).
\]
After taking $s_1 \rightarrow 0$, we obtain that
\begin{align}\label{I1}
{\rm Re} \int_0^{\xi} \sum\limits_{j=1}^2 \int_{\Gamma_j} \mu_j^{-1} \mathscr
T_j U  \bar \psi_1 {\rm d}\gamma_j {\rm d}t \leq 0.
\end{align}

For $0 \leq t \leq \xi \leq T,$  we have from (\ref{AP1}) that
\begin{align}\label{I2}
{\rm Re }\int_0^{\xi} \int_{\Gamma_1} \rho  \bar \psi_1  {\rm d}  \gamma_1 {\rm d} t
&=\int_0^{\xi} \big(  \int _0^{t} \int _{\Gamma_1} \rho (\tau) {\rm d} \gamma_1
{\rm d}  \tau\big)  \bar U  {\rm d} t \nonumber\\
&\leq \int_0^{\xi} \int_0 ^{t} \| \rho (\cdot, \tau)\|_{H^{-1/2}(\Gamma_1)} \|
U(\cdot, t)\|_{H^{1/2} (\Gamma_1)} {\rm d} \tau {\rm d}  t \notag\\
&\lesssim \int_0^{\xi} \int_0 ^{t} \| \rho (\cdot, \tau)\|_{H^{-1/2}(\Gamma_1)}
\| U (\cdot, t)\|_{H_{\rm p}^{1} (\Omega)} {\rm d} \tau {\rm d} t \notag\\
&\leq \big( \int_0^{\xi} \| \rho (\cdot, t)\|_{H^{-1/2}(\Gamma_1)}  {\rm d} t 
\big) \big(\int_0^{\xi}\|U(\cdot, t)\|_{H_{\rm p}^1 (\Omega)}  {\rm d} t \big).
\end{align}
Using integration by parts and ({\ref{AP}}), we have
\begin{align*}
\int_0^{\xi}
 &\int_\Omega \mu^{-1} \partial_t(\partial_x U)  \bar \psi_1 {\rm d}x {\rm d}z
{\rm d}t+ \int_0^{\xi} \int_{\Omega} \partial_x (\mu^{-1} \partial_t U)  \bar
\psi_1 {\rm d}x {\rm d}z {\rm d}t\\
&=\int_{\Omega} \big(\mu^{-1} \partial_x U 
\bar \psi_1\big) \mid_0^{\xi} {\rm d} x  {\rm d} z -
\int_0^{\xi} \mu^{-1} \partial_x U  \partial_t \bar \psi_1 {\rm d} t {\rm d} x {\rm d}z\\
&+\int_{\Omega} \partial_x (\mu^{-1} U) \cdot \bar \psi_1 \mid_0^{\xi} {\rm d} x {\rm d} z
- \int_0^{\xi} \partial_x(\mu^{-1} U) \cdot \partial_t \bar \psi_1 {\rm d} x {\rm d} z {\rm d}t\\
&=\int_0^{\xi} \int_{\Omega}  \big(\mu^{-1} \partial_x U +\partial_x (\mu^{-1} U)\big) \cdot \bar U {\rm d}x {\rm d}z {\rm d}t.
\end{align*}
By the periodicity of $\mu$ and $U$ in $x$, it yields that
\[
\int_0^{\xi} \int_{\Omega}  \big(\mu^{-1} \partial_x U +\partial_x (\mu^{-1} U)\big)  \bar U {\rm d} x {\rm d} z {\rm d}t
+\int_0^{\xi} \int_{\Omega}  \big(\mu^{-1} \partial_x  \bar U +\partial_x
(\mu^{-1} \bar U)\big) U {\rm d} x {\rm d} z {\rm d}t=0.
\]
Thus
\begin{equation}\label{I3}
{\rm Re}\int_0^{\xi}  \int_{\Omega}(\mu^{-1} \partial_{tx} U+
\partial_x (\mu^{-1} \partial_t U)) \bar \psi_1 {\rm d} x {\rm d} z {\rm d}t=0.
\end{equation}
Substituting (\ref{I1})--(\ref{I3}) into (\ref{EE1}), we have for any $\xi \in
[0, T]$ that
\begin{align}\label{EE2}
\frac{1}{2}
&\| (\varepsilon -c_1^2\mu^{-1})^{1/2} U(\cdot, \xi)\|^2_{L^2
(\Omega)}+\frac{1}{2}\int_{\Omega}\mu^{-1}|\int_0^{\xi} \nabla U(\cdot, t) {\rm
d}t|^2 {\rm d} x {\rm d} z  \nonumber\\
&\leq \big( \int_0^{\xi} \| \rho (\cdot, t)\|_{H^{-1/2}(\Gamma_1)}  {\rm d} t 
\big) \big(\int_0^{\xi}\|U(\cdot, t)\|_{H_{\rm p}^1 (\Omega)} \big) {\rm d} t.
\end{align}

Taking the derivative of (\ref{TRP}) with respect to $t$, we know that
$\partial_t U$ satisfies the same equation with $\rho$ replaced by $\partial_t
\rho$.  Define 
\[
\psi_2 (x, z, t) =\int_t ^{\xi} \partial_t U (x, z, \tau) {\rm d} \tau, \quad(x,
z ) \in \Omega, ~ 0 \leq t  \leq \xi.
\]
We may follow the same steps as those for $\psi_1$ to obtain
\begin{align}\label{EE3}
\frac{1}{2}
&\| (\varepsilon -c_1^2\mu^{-1})^{1/2} \partial_t U(\cdot,
\xi)\|^2_{L^2 (\Omega)}+\frac{1}{2}\int_{\Omega}\mu^{-1}|\int_0^{\xi} \partial_t
(\nabla U(\cdot, t)) {\rm d} t|^2  {\rm d} x {\rm d}z \nonumber\\
=&{\rm Re} \int_0^{\xi} \sum\limits_{j=1}^2 \int_{\Gamma_j} \mu_j^{-1} \mathscr
T_j \partial_t U  \bar \psi_2 {\rm d}\gamma_j {\rm  d} t
+{\rm Re }\int_0^{\xi} \int_{\Gamma_1} \partial_t \rho  \bar \psi_2 {\rm d} \gamma_1 {\rm d} t \nonumber \\
&-c_1 {\rm Re}\int_0^{\xi}  \int_{\Omega}(\mu^{-1} \partial_{ttx} U+
\partial_x (\mu^{-1} \partial^2_t U)) \bar \psi_2 {\rm d} x  {\rm d} z {\rm d} t.
\end{align}
Integrating by parts yields that
\begin{equation}\label{I4}
\frac{1}{2}\int_{\Omega}\mu^{-1}|\int_0^{\xi} \partial_t (\nabla U(\cdot, t))
{\rm d} t|^2 {\rm d} x  {\rm d} z=\frac{1}{2} \| \mu^{-1/2} \nabla U(\cdot,
\xi)\|^2_{L^2(\Omega)}.
\end{equation}
The first and the third terms on the right-hand side of (\ref{EE3}) are
discussed as above.  We only have to consider the second term. By (\ref{AP}),
Lemma \ref{TT}, and Lemma \ref{DTN}, we get
\begin{align}\label{I5}
\int_0^{\xi}\int_{\Gamma_1} \partial_t\rho  \bar \psi_2 {\rm d} \gamma_1 {\rm d} t
 &=\int_{0}^{\xi} \int_0^t (\int_{\Gamma_1} \partial_\tau \rho (\cdot, \tau){\rm
d} \gamma_1) {\rm d}\tau  \partial_t \bar U (\cdot, t) {\rm d} t \nonumber\\
&=\int_{\Gamma_1} \big(\int_0^{t} \partial_{\tau} \rho (\cdot, \tau) {\rm d}
\tau\big)\bar U (\cdot, t)\mid_0^{\xi} {\rm d} \gamma_1
-\int_0^ \xi \int_{\Gamma_1} \partial_t \rho (\cdot, t)  U(\cdot, t) {\rm d} \gamma_1 {\rm d} t \nonumber\\
& \lesssim \int_0^{\xi } \| \partial_t \rho (\cdot,
t)\|_{H^{-1/2}(\Gamma_1)}\| U(\cdot, t)\|_{H^{1/2}(\Gamma_1)}
{\rm d} t \nonumber\\
& \lesssim \int_0^{\xi }\| \partial_t \rho (\cdot, t)\|_{H^{-1/2}(\Gamma_1)}
\| U(\cdot, t)\|_{H_{\rm p}^{1}(\Omega)} {\rm d}t.
\end{align}
Substituting (\ref{I4}) and (\ref{I5})  into (\ref{EE3}), we have for any $\xi \in [0, T]$ that
\begin{align}\label{EE4}
\frac{1}{2}\|(\varepsilon -c_1^2\mu^{-1})^{1/2} \partial_t
U(\cdot, \xi)\|^2_{L^2 (\Omega)}
+\frac{1}{2} \| \mu^{-1/2} \nabla U(\cdot, \xi)\|^2_{L^2(\Omega)}
\nonumber \\
\lesssim \int_0^{\xi } \| \partial_t \rho (\cdot, t)\|_{H^{-1/2}(\Gamma_1)}
\|U(\cdot, t)\|_{H_{\rm p}^{1}(\Omega)} {\rm d}t.
\end{align}
Combing the estimates (\ref{EE2}) and (\ref{EE4}), we obtain 
\begin{align}\label{UH1}
\|U(\cdot, \xi)\|^2_{L^2(\Omega)}+\| \nabla U(\cdot, \xi)\|^2_{L^2(\Omega)^2}
\lesssim
 \big( \int_0^{\xi} \| \rho (\cdot, t)\|_{H^{-1/2}(\Gamma_1)} {\rm d }t  \big)
\big(\int_0^{\xi}\|U(\cdot, t)\|_{H_{\rm p}^1 (\Omega)} {\rm d }t\big)
\nonumber\\
+\int_0^{\xi } \| \partial_t \rho (\cdot,
t)\|_{H^{-1/2}(\Gamma_1)} \| U(\cdot, t)\|_{H_{\rm p}^{1}(\Omega)} {\rm d}t.
\end{align}
Taking the $L^{\infty}$- norm with respect to $\xi$ on both side of (\ref{UH1}) yields
\begin{align*}
\| U\|^2_{L^{\infty} (0, T;~ L^2(\Omega))}+\| \nabla U\|^2_{L^{\infty}(0,
T; L^2(\Omega)^2)}\lesssim T \| \rho\|_{L^1(0, T; H^{-1/2}
(\Gamma_1))}\|U\|_{L^{\infty}(0,T;  H^1_{\rm p}(\Omega))}\\
+\|\partial_t \rho \|_{L^1(0, T; H^{-1/2}
(\Gamma_1))}\|U\|_{L^{\infty}(0,T; H^1_{\rm p}(\Omega))},
\end{align*}
which gives the estimate (\ref{ES1}) after applying the Young inequality.

Integrating (\ref{UH1}) with respect to $\xi $ from $0$ to $T$ and using the
Cauchy--Schwarz inequality, we obtain
\begin{align*}
\|U\|^2_{L^2(0, T;~ L^2(\Omega))}+\| \nabla U\|^2_{L^2 (0, T;~ L^2 (\Omega)^2)}
\lesssim T^{3/2} \| \rho\|_{L^1(0, T; H^{-1/2} (\Gamma_1))} \| U\|_{L^2(0,
T; H^1_{\rm p}(\Omega)})\\
+T^{1/2} \| \partial_t \rho\|_{L^1(0, T; H^{-1/2} (\Gamma_1))}  \|
U\|_{L^2(0, T; H^1_{\rm p}(\Omega)}),
\end{align*}
which implies the estimate (\ref{ES2}) by using the Young inequality again.
\end{proof}

\section{Conclusion}\label{CL}

In this paper, we studied the time-domain scattering problem in a
one-dimensional grating. The TE and TM cases were considered in a unified
approach. The scattering problem was reduced equivalently into
an initial-boundary value problem in a bounded domain by using the exact
time-domain DtN map. The reduced problem was shown to have a unique solution by
using the energy method. The stability was also presented. The main
ingredients of the proofs were the Laplace transform, the Lax--Milgram
lemma, and the Parseval identity. Moreover, by directly considering the
variational problem of the time-domain wave equation, we obtained a priori
estimates with explicit dependence on time. In the future, we plan to
investigate the time-domain scattering by biperiodic structures where the full
three-dimensional Maxwell's equations should be considered. The progress will
be reported elsewhere.

\end{document}